%% file: main.tex
\documentclass[a4paper,11pt]{amsart}

\usepackage{amsmath}
\usepackage{amssymb}
\usepackage{color}
\usepackage{enumerate}
\usepackage{amsthm}
\usepackage{listings}
\usepackage[all]{xy}
\usepackage{longtable}
\usepackage{graphicx}
\usepackage{hyperref}

\newtheorem{theorem}{Theorem}[section]
\newtheorem{proposition}[theorem]{Proposition}
\newtheorem{lemma}[theorem]{Lemma}
\newtheorem{corollary}[theorem]{Corollary}
\theoremstyle{definition}
\newtheorem{definition}[theorem]{Definition}
\theoremstyle{remark}
\newtheorem{remark}[theorem]{Remark}
\newtheorem{example}[theorem]{Example}

\newcommand{\Mat}{\operatorname{Mat}}

\newcommand{\rank}{\operatorname{rank}}
\newcommand{\codim}{\operatorname{codim}}

\newcommand{\D}{\operatorname{d}\!}
\newcommand{\CC}{\mathbb{C}}

\newcommand{\NN}{\mathbb{N}}

\newcommand{\PP}{\mathbb{P}}
\newcommand{\OO}{\mathcal{O}}

\title{Bouquet decomposition for Determinantal Milnor fibers}
\author{Matthias Zach, Johannes Gutenberg Universit\"at Mainz}
\address{Institut f\"ur Mathematik \\
FB 08 - Physik, Mathematik und Informatik \\
Johannes Gutenberg-Universit\"at \\
Staudingerweg 9, 4. OG \\
55128 Mainz, Germany}
\begin{document}

\maketitle

\input{results.tex}

\input{preliminaries.tex}

\input{proof_of_the_main_theorem.tex}

\bibliographystyle{amsplain}
\bibliography{sources}

\end{document}

%% file: results.tex
\section{Results}

In this note we will apply a general Bouquet Decomposition Theorem 
by M. Tib{\u a}r \cite{Tibar95} in the case of an Essentially Isolated 
Determinantal Singularity (EIDS, see \cite{EbelingGuseinZade09}) 
to prove the following:

\begin{theorem}
  Let $(X_0,0) = (A^{-1}(M_{m,n}^t),0) \subset (\CC^N,0)$ be an EIDS 
  of type $(m,n,t)$ given by a holomorphic map germ 
  \[
    A \colon (\CC^N,0) \to (\Mat(m,n;\CC),0),
  \]
  $A_u$ a stabilization of $A$, 
  and $\overline X_u = A_u^{-1}(M_{m,n}^t)$ the determinantal 
  Milnor fiber. 
  Define 
  \[
    s_0 := \min \{ s \in \NN : (m-s+1)(n-s+1) \leq N \}.
  \]
  Then $\overline X_u$ has a bouquet decomposition as 
  \begin{equation}
    \overline X_u \cong_{ht} L_{m,n}^{t,m\cdot n -N} \vee 
    \bigvee_{s_0 \leq s \leq t} \bigvee_{i=1}^{r(s)} S^{N-(m-s+1)(n-s+1)+1}( 
      L_{m-s+1,n-s+1}^{t-s,1}
      )
    \label{eqn:BouquetDecomposition}
  \end{equation}
  where $L_{m,n}^{t,k}$ is the intersection 
  \[
    L_{m,n}^{t,k} := H_k \cap M_{m,n}^t
  \]
  of a codimension $k$ hyperplane $H_k$ in general position off the origin 
  with the generic determinantal variety
  \[
    M_{m,n}^t = \{ M \in \Mat(m,n;\CC) : \rank M < t \}.
  \]
  \label{thm:MainTheorem}
\end{theorem}

In the formula (\ref{eqn:BouquetDecomposition}) we denote by 
$S^r(X)$ the $r$-fold repeated suspension of a topological 
space $X$. 
We use the same convention as in \cite{Tibar95} and set 
  $S^1(\emptyset) = S^0$,
the sphere of dimension $0$, and 
  $S^0(X) = X$
for any $X$. 

Theorem \ref{thm:MainTheorem} is a major reduction step in the understanding 
of the vanishing topology of essentially isolated determinantal singularities. 
For isolated complete intersection singularities (ICIS) on a smooth ambient 
space $(\CC^N,0)$, it was already known 
since \cite{Hamm72} that the Milnor fiber is homotopy equivalent to a bouquet 
of spheres of the same dimension. This is not necessarily the case for 
EIDS, see for example \cite{DamonPike14}. Several groups have studied their 
vanishing Euler characteristic, see e.g. \cite{EbelingGuseinZade09}, 
\cite{GaffneyNivaldoRuas16}, and \cite{BallesterosOreficeTomazella13}. 
One approach is to study the behaviour of a generic hyperplane 
equation $h$ in a determinantal deformation of a given EIDS $(X_0,0)$. 
The determinantal Milnor fiber $\overline X_u$ is then obtained from 
its hyperplane section $\overline X_u \cap \{ h= 0\}$ by attaching 
cells, or more generally in the context of stratified Morse theory
so-called ``thimbles\footnote{By a thimble we mean the pair of topological 
spaces given by the product of the tangential and the normal Morse data at a given 
critical point. This might differ from the cell $(D^\lambda, \partial D^\lambda)$
occuring in classical Morse theory, see \cite{GoreskyMacPherson88}.}'', 
at Morse critical points 
of $h$ on $\overline X_u$. This way one obtains nice formulas for the vanishing 
Euler characteristic in terms of the polar multiplicities of the singularity 
$(X_0,0)$. However, it is hardly possible to describe the loci in 
the hyperplane section $\overline X_u \cap \{h = 0\}$, at which the attachments 
take place. This fact destroys any hope to arrive at a precise description of the 
homotopy type of $\overline X_u$. 

It is the Carrousel by L\^e, which sits at the heart of the proof of the 
Handlebody Theorem (stated as Theorem \ref{thm:HandlebodyTheorem} below)
from \cite{Tibar95}, and 
enables us to understand the attachments of the thimbles. As we will see, however, 
the setup for the application of the Handlebody Theorem is quite different 
from the viewpoint of EIDS. We will describe the transformation of any EIDS
$(X_0,0) = (A^{-1}(M_{m,n}^t) \subset (\CC^N,0)$
to an ICIS 
\[
  (X_0,0) = (\{f_{1,1}= \cdots = f_{m,n} = 0\},0) \subset (Z,0)
\]
on a controlled Whitney stratified ambient space 
\[
  (Z,0) \cong (M_{m,n}^t,0) \times (\CC^N,0)
\]
in Section \ref{sec:TheStandardTransformation}. 
Then, rather than doing an induction argument by cutting down with generic 
hyperplanes, we proceed by an inductive argument, where we always trade 
one equation $f_{i,j}$ defining $(X_0,0)$ in $(Z,0)$ for a generic hyperplane
equation and eventually end up with the space $L_{m,n}^{t,m\cdot n -N}$ -
a generic linear section of $M_{m,n}^t$.
During this process, the Handlebody Theorem allows us to really keep track 
of the involved attachment processes. 

The homotopy type of the spaces $L_{m,n}^{t,k}$ has been studied in a 
few particular cases, see e.g. \cite{FKZ15}. The Euler obstructions of the 
generic determinantal varieties $M_{m,n}^t$, which are closely related to 
their hyperplane sections $L_{m,n}^{t,1}$, can be found in \cite{GaffneyNivaldoRuas16} 
and the Chern-Schwartz-MacPherson classes of their projectivizations 
$\PP(M_{m,n}^t)$ have been studied in 
\cite{Zhang16}. However, there is - at least to the 
knowledge of the author - no complete understanding of the $L_{m,n}^{t,k}$
for arbitrary values of $m,n,t$, and $k$. 

\medskip

I would like to thank in chronological order: J.J. Nu\~no Ballesteros for 
his kind hospitality during a stay in Valencia, where I first heard of the 
Carrousel, D. van Straten for conversations leading to the study of the work 
by M. Tib{\u a}r, H. Hamm for his encouraging interest in the subject and 
discussion during a workshop on determinantal singularities in Hannover, 
and finally M. Tib{\u a}r himself for an exchange of letters on how to 
use his results for ICIS.

%% file: preliminaries.tex
\section{Preliminaries}

\subsection{Notations and Background}

In this article we will make use of the common terms of stratified Morse 
theory. The reader may consult the standard textbook reference 
\cite{GoreskyMacPherson88}. Given any space $Z$ with 
a Whitney stratification $\Sigma = (S_{\alpha})_{\alpha \in A}$ and a point 
$p \in Z$ we will write 
\[
  T_p Z := T_p S_\alpha
\]
for the tangent space of $S_\alpha$ at $p$
where $S_\alpha$ is the stratum containing $p$. 

Suppose $Z \subset N$ is embedded in a smooth manifold $N$. We say that 
a smooth map 
\[
  f \colon M \to N \supset Z
\]
is transverse to $Z$ in a point $p\in M$, if 
\begin{itemize}
  \item either $\dim M + \dim Z < \dim N$ and $f(p) \notin Z$, 
  \item or $f(p) \in Z$ and $\D f(p) T_p M + T_{f(p)}Z = T_{f(p)} N$.
\end{itemize}
The map $f$ is transverse to $Z$ in a set $U \subset M$, if it is transverse 
to $Z$ at every point $p\in U$.

Consider the set $X = f^{-1}(Z)$. It is naturally stratified by the strata 
$\Sigma_\alpha = f^{-1}(S_\alpha)$ given by the preimage of the strata of $Z$.
Whenever $f \colon M \to N \supset Z$ is transverse to $Z$ in $M$, the 
$\Sigma_{\alpha}$ are a Whitney stratification for $X$ and we also say 
that $X$ \textit{inherits} the stratification of $Z$. 
In particular, this applies to the case of a closed embedding such as 
for example the fiber of a stratified submersion on $Z$ induced from a 
map on $N$.

For a matrix $A \in \Mat(m,n;R)$ with entries in a commutative ring $R$ 
we will denote by $\langle A \rangle$ the ideal in $R$ generated by its entries. 
For $0< t \leq \min\{m,n\}$ we will write 
\[
  A^{\wedge t} \colon \bigwedge^t R^m \to \bigwedge^t R^n
\]
for the map induced by $A \colon R^m \to R^n$ on the $t$-th exterior 
products. Then $\langle A^{\wedge t} \rangle$ is the ideal generated by the $t$-minors 
of $A$.

On the topological side let us emphasize that we usually consider 
\textit{closed} Milnor balls $B$ for singularities. This convention 
always assures that one automatically keeps track of the boundary behaviour in 
deformations, which can be a particularly tricky task in the setting of 
non-isolated singularities. Moreover, the resulting Milnor fibers are always 
\textit{compact} stratified spaces, which simplifies their treatment by
Morse theory.

Since this note is merely an application of methods which had been developed 
before, we will restrict ourselves to the description of how the techniques 
can be used on determinantal singularities. To this end, we will describe the 
cornerstones of the proofs of e.g. the Handlebody Theorem by Tib{\u a}r and 
other ideas behind it. However, the reader who is unfamiliar with the mathematical 
rigor on singularity theory on Whitney stratified spaces is strongly encouraged 
to consult the articles \cite{Tibar95}, \cite{Le77}, and the references given 
there, and the standard textbook on stratified Morse theory 
\cite{GoreskyMacPherson88}.

\subsection{Essentially Isolated Determinantal Singularities}
  
Let $(M_{m,n}^t,0) \subset (\Mat(m,n;\CC),0)$ be the generic determinantal 
variety of type $(m,n,t)$:
\[
  M_{m,n}^t = \{ M \in \Mat(m,n;\CC): \rank M < t \}.
\]
The canonical rank stratification 
by 
\[
  S_{m,n}^s = M_{m,n}^s \setminus M_{m,n}^{s-1}
\]
for $0 \leq s \leq \min\{m,n\}$ is a Whitney 
stratification of $\Mat(m,n;\CC)$ and $M_{m,n}^t$. This can easily be deduced 
by induction from the observation that at any point $p \in S_{m,n}^s$ one has 
a product 
\begin{equation}
  (M_{m,n}^t,p) \cong (M_{m-s+1,n-s+1}^{t-s},0) \times (\CC^{(m+n)\cdot (s-1) - (s-1)^2},0)
  \label{eqn:LocalProductGenericDeterminantalVariety}
\end{equation}
of analytic spaces. 
We will denote by 
\begin{equation}
  L^{t,k}_{m,n}
  \label{eqn:DefinitionL-thComplexLink}
\end{equation}
the $k$-th complex link of $M_{m,n}^t$, 
that is the interior of the Milnor fiber of the complete intersection morphism 
given by a generic linear map
\[
  \Phi_k : (\Mat(m,n;\CC),0) \to (\CC^k,0).
\]
The usual complex link is then just $L^{t,1}_{m,n}$. Note that, since the 
generic determinantal variety $M_{m,n}^t$ has a homogeneous singularity at the 
origin, we may choose the Milnor ball to be of arbitrary - even infinite - size. 
Thus, up to homeomorphism 
this definition agrees with the one given in the statement of Theorem 
\ref{thm:MainTheorem}

The complex link plays a central role in the stratified Morse theory 
on complex analytic varieties, since it forms the ``normal Morse data'', 
see \cite{GoreskyMacPherson88}. In the case of the generic determinantal 
variety $M_{m,n}^t$ we find from (\ref{eqn:LocalProductGenericDeterminantalVariety}) 
that the normal Morse data along the stratum $S_{m,n}^s$ 
for $s\leq t$ is given by the pair of spaces
\begin{equation}
  \left( C(L_{m-s+1,n-s+1}^{t-s,1}), L_{m-s+1,n-s+1}^{t-s,1} \right),
  \label{eqn:NormalMorseDataGenericDeterminantalVariety}
\end{equation}
where $C(X)$ denotes the real cone over a given topological space $X$.

\begin{definition}[\cite{EbelingGuseinZade09}]
  An EIDS $(X_0,0) \subset (\CC^N,0)$ of type $(m,n,t)$ is given as 
  $(X_0,0) = (A^{-1}(M_{m,n}^t),0)$, where $A : (\CC^N,0) \to (\Mat(m,n;\CC),0)$ 
  is a holomorphic map germ, for which $A$ is 
  transversal to the rank stratification of $\Mat(m,n;\CC)$
  in a punctured neighborhood of the origin 
  and $\codim (X,0) = \codim M_{m,n}^t = (m-t+1)(n-t+1)$. 
\end{definition}

It follows directly from the definition of transversality that away 
from the origin also 
$X_0$ inherits a canonical stratification by the strata 
\[
  \Sigma^s := A^{-1}(S_{m,n}^s).
\]
Counting dimensions yields that 
these strata are nonempty if and only if  
\begin{equation}
  \min \{ r \in \NN : (m-r+1)(n-r+1) \leq N \} < s \leq t.
  \label{eqn:RangeOfNonemptyStrata}
\end{equation}

An ``essential smoothing'' of $(X_0,0)$ is a family 
\[
  \xymatrix{
    X_0 \ar@{^{(}->}[r] \ar[d] &
    X \ar[d]_u \\
    \{0 \} \ar[r] & 
    \CC
  }
\]
coming from a stabilization 
\[
  \mathbf A : (\CC^{N},0) \times (\CC,0) \to (\Mat(m,n;\CC),0) \times (\CC,0)
\]
of the map $A$. That is $\mathbf A = \mathbf A(x,u) = (A_u(x),u)$ with 
$A_0 = A$ 
and $A_u$ is transversal to $M_{m,n}^t$ for all $u\neq 0$ in a sufficiently small.
Then, the total space of the family above appears as 
$X = \mathbf A^{-1}(M_{m,n}^t\times \CC)$ 
and $u$ is map given by the deformation parameter. 

From a stabilization we can construct the \textit{determinantal Milnor fiber} 
as follows. Choose a representative 
\[
  \mathbf A : W \times U \to \Mat(m,n;\CC) \times U
\]
of the stabilization $\mathbf A$ for some open sets 
$W\subset \CC^N$ and $U \subset \CC$ and let $B \subset \CC^N$ be a Milnor ball 
for $(X_0,0)$ in $W$. By this we mean a closed ball around the origin such that 
$\overline X_0 := X_0 \cap B$ is closed, the boundary $\partial B$ intersects 
$X_0$ transversally, and 
\[
  \overline X_0 \cong C(\partial \overline X_0)
\]
is homeomorphic to the real cone over its boundary 
$\partial \overline X_0 = \partial B \cap X_0$. 
We can then consider the family $u : X \cap (B \times U) \to U$. 
It may be deduced from Thom's first Isotopy Lemma that 
$u$ is a trivial topological fibration along the boundary 
$X \cap \partial B\times U$ over $U$ and that 
\[
  u : (X \cap (B \times U)) \setminus \overline X_0 \to U\setminus\{0\}
\]
is a topological fiber bundle for $U$ small enough.

\begin{definition}
  \label{def:DeterminantalMilnorFiber}
  It is the fiber of this bundle 
  \[
    \overline X_u \cong A_u^{-1}(M_{m,n}^t) \cap B 
  \]
  that we call the determinantal Milnor fiber. 
\end{definition}

Using the theory of versal unfoldings, one can show that 
in fact for any given EIDS $(X_0,0)$ the determinantal Milnor fiber 
is unique up to homeomorphism, see \cite{Pereira10} or \cite{Zach17}.

\begin{example}
  Consider the EIDS $(X_0,0) \subset (\CC^5,0)$ of type $(2,3,2)$ given by the matrix
  \[
    A = 
    \begin{pmatrix}
      x & y & z \\
      v & w & x
    \end{pmatrix}
  \]
  together with the essential smoothing induced by the perturbation with
  \[
    \begin{pmatrix}
      u & 0 & 0 \\
      0 & 0 & -u 
    \end{pmatrix}.
  \]
  It is easily seen that indeed the total space $(X,0) \subset (\CC^{5+1},0)$ 
  is isomorphic to the generic determinantal variety 
  $M_{2,3}^2 \subset \Mat(2,3;\CC) \cong \CC^6$ and the map $u$ is 
  a generic linear form on it. Hence, the determinantal 
  Milnor fiber of $(X_0,0)$ is nothing but the (closure of the) 
  complex link $L_{2,3}^{2,1}$ of 
  $(M_{2,3}^2,0)$. It is known that $L_{2,3}^{2,1}$ is homotopy equivalent 
  to the $2$-sphere $S^2$, see \cite{FKZ15}.
\end{example}

Browsing through the classification of simple isolated Cohen-Macaulay 
codimension $2$ singularities from \cite{FruehbisKruegerNeumer10}, 
one sees that similar results hold for the first entries of the lists 
for the different dimensions $d=0,\dots,3$. For $d<3$, the determinantal 
Milnor fibers are, however, hyperplane sections of the link 
$L_{2,3}^{2,1}$. 
This explains our interest in the $k$-th links of the generic determinantal 
singularities rather than just their usual links.

\subsection{The Handlebody Theorem}
In \cite{Tibar95}, M. Tib{\u a}r proofs the following theorem 
for the Milnor fiber $F$ of an isolated hypersurface singularity 
\[ 
  f \colon (Z,0) \to (\CC,0)
\]
on a complex analytic, Whitney stratified space $(Z,0)$ and 
the complex link $L$ of $(Z,0)$:

\begin{theorem}[\cite{Tibar95}, Handlebody Theorem]
  The Milnor fiber $F$ is obtained from the complex link $L$ 
  to which one attaches cones over local Milnor fibres of stratified 
  Morse singularities. The image of each such attaching map retracts 
  within $L$ to a point.
  \label{thm:HandlebodyTheorem}
\end{theorem}

We give a rough outline of the idea of the proof. 
We may assume $(Z,0) \subset (\CC^N,0)$ to be embedded in some smooth ambient 
space. 
Let $h$ be the linear equation on $\CC^N$ defining the link $L$ of $(Z,0)$ and consider 
\begin{equation}
  \Phi = (h,f) : B \cap Z \cap \Phi^{-1}( D \times D' ) \to D \times D'
  \label{eqn:LeFibration}
\end{equation}
for a sufficiently small, closed ball $B$ and discs $D, D' \subset \CC$ around 
the origin. In \cite{Le77}, L\^e has shown the following. 
There exists a Zariski open set $\Omega \subset \left( \CC^N \right)^\vee$ 
of linear forms on the ambient space such that for $h \in \Omega$
the polar variety 
\[
  \Gamma(h,f) := 
  \overline{\{ z \in Z\setminus f^{-1}(\{0\}) : 
    \exists a \in \CC: \D h(z)|_{T_z Z} = a \cdot \D f(z)|_{T_z Z} \} },
\]
i.e. the critical locus of $h$ on $Z$ relative to $f$,
is a curve, which is branched over its image 
\[
  \Delta = \Delta(h,f) = \Phi(\Gamma(h,f)) \subset D \times D', 
\]
the so-called Cerf-diagram.
The proof for the set $\Omega$ of admissable hyperplane equations to 
be Zariski open can be found in \cite{HammLe73}. 
Moreover, one can choose $D'$ small enough such that the intersection 
$\Delta \cap \partial D \times D'$ is empty. 
Then $\Phi$ is a topological fibration away from $\Delta$ and 
one has homeomorphisms 
\[
  F \cong \Phi^{-1}( D \times \{\delta\}) 
\]
and 
\[
  L \cong \Phi^{-1}( \{\eta\}\times D' )
\]
for $0 \neq \delta$, resp. $0 \neq \eta$, sufficiently small. 
It is also shown in \cite{HammLe73} that $\Omega$ can be chosen such that the restriction 
of $h\in \Omega$ to 
any fixed fiber $\Phi^{-1}(D \times \{\delta\})$
has only Morse singularities over the intersection points 
$\Delta \cap D \times \{\delta\}$ 
for $0 \neq \delta \in D'$.

At this point the so-called ``Carrousel'' is furnished by the 
geometric monodromy of $F$ along the boundary of $D'$, i.e. by 
the variation of the value $\delta$ of $f$, but one does not 
only construct a lifting of the unit tangent vector field 
along $\partial D'$ to 
$\Phi^{-1}(D \times \partial D')$, but one also keeps track 
of the monodromy induced on the disc 
$D \times \{\delta\}$, the intersection points 
$C = \Delta(h,f) \cap D \times \{\delta\}$, and the corresponding 
critical points of $h$ on the Milnor fiber $\Phi^{-1}(D\times\{\delta\})$ 
over them. 

Let $F' = \Phi^{-1}(\{(\eta,\delta)\})$. Then up to homotopy 
the Milnor fiber $F$ is obtained from $F'$ by attaching  
thimbles along suitably chosen paths in $D \times \{\delta\}$ 
from $(\eta,\delta)$ 
to the critical values of the stratified Morse points of $h$ on $F$. 
The topology of each of 
these attachments is governed by the Morse data. In the situations 
we will encounter in the context of EIDS, the Morse data 
will always be of the following 
form:

\begin{proposition}
  \label{prp:ThimbleForDeterminantalSingularities}
  Let $(X,p) \cong (M_{m,n}^s,0) \times (\CC^k,0)$ and $h : (X,p)\to (\CC,0)$ a 
  holomorphic map germ with a stratified Morse singularity at $p$. Then the 
  thimble corresponding to this critical point is 
  \[
    (C(S^k(L_{m,n}^{s,1})),S^k(L_{m,n}^{s,1}) ),
  \]
  i.e. one attaches the real cone $C(S^{k}(L_{m,n}^{s,1}))$ along its boundary 
  $S^k(L_{m,n}^{s,1})$. 
\end{proposition}

The key observation from the Carrousel is that keeping track of the relative 
critical points of the hyperplane equation $h$ on $F$ allows one to determine 
exactly at which locus on $F'$ these attachments take place. 

As a final step, one constructs another homeomorphism 
$L \cong \Phi^{-1}(W) \subset F$ on a certain subspace $\Phi^{-1}(W)$ of $F$ 
by ``sliding along $\Delta$''. The space $W$ is chosen such that 
$F' \subset \Phi^{-1}(W)$ and one can use the carrousel monodromy to show that 
for each thimble $e$ one has to attach to $\Phi^{-1}(W)$ to complete it 
- up to homotopy - to $F$, there is already one thimble $e'$ that had been attached 
to $F'$ in the same spot as $e$ to get $\Phi^{-1}(W)$. This explains, why each 
attaching map in the statement of the Handlebody Theorem \ref{thm:HandlebodyTheorem} 
retracts within $L$ to a point.

%% file: proof_of_the_main_theorem.tex
\section{Proof of the Main Theorem}

\subsection{The Standard Transformation}
\label{sec:TheStandardTransformation}

Let 
$(X_0,0) \subset (\CC^N,0)$ be an EIDS of type $(m,n,t)$ given by 
a matrix $A$. 
In this section we will explain how to transform $(X_0,0)$ and its 
determinantal deformations into an ICIS on a canonical ambient 
space $(Z,0) \cong (M_{m,n}^t,0) \times (\CC^N,0)$. 

Let $Y = \Mat(m,n;\CC) \cong \CC^{m\cdot n}$, 
\[
  \mathbf Y =
  \begin{pmatrix}
    y_{1,1} & \cdots & y_{1,n} \\
    \vdots & & \vdots \\
    y_{m,1} & \cdots & y_{m,n} 
  \end{pmatrix}
\]
the tautological matrix with the standard coordinates $y_{i,j}$ as entries, 
$\CC[\underline y]$ the associated affine coordinate ring, 
and $\OO_{m\cdot n} = \CC\{\underline y \}$ the local ring of $(Y,0)$. 

Let $A : U \to Y$ be a representative of the matrix $A$ defining $(X_0,0)$.
We set 
\[
  Z := 
  \left\{ (x,y) \in U \times Y : \left(A(x) - \mathbf Y(y)\right)^{\wedge t} = 0  \right\}.
\]
This space comes along with a commutative diagram and two natural projections 
\begin{equation}
  \xymatrix{ 
    X_y \ar@{^{(}->}[r] \ar[d] &
    Z \ar[d]_p \ar[r]^q & U \\
    \{y\} \ar@{^{(}->}[r] & Y & \\
  }.
  \label{eqn:ProjectionsFromZ}
\end{equation}
For any point $y \in Y$ we can view the fiber 
\[
  X_y = \left( A - \mathbf Y(y) \right)^{-1}(M_{m,n}^t) = q\left( p^{-1}(\{y\}) \right)
\]
as a determinantal deformation of the EIDS $(X_0,0)$.

It follows from the obvious change of coordinates
\begin{equation}
  \tilde x_k = x_k, \quad 
  \tilde y_{i,j} = y_{i,j} - a_{i,j}(x)
  \label{eqn:CoordinateChangeAmbientSpace}
\end{equation}
on the ambient space 
$(\CC^N,0) \times (Y,0)$ that we have an isomorphism
\begin{equation}
  (Z,0) \cong (M_{m,n}^t,0) \times (\CC^N,0).
  \label{eqn:ProductDecompositionForZ}
\end{equation}
In particular, $(Z,0)$ enjoys a canonical Whitney stratification by the strata
\[
  (\tilde S_{m,n}^s,0) = (S_{m,n}^s,0) \times (\CC^N,0)
\]
inherited 
from the rank stratification on $M_{m,n}^t$.


\begin{proposition}
  The projection $q : (Z,0) \to (Y,0)$ is a complete intersection morphism 
  with isolated singularity on $(Z,0)$ relative to its canonical 
  Whitney stratification. In this sense 
  \[
    \xymatrix{
      (X_0,0) \ar[d] \ar@{^{(}->}[r] &
      (Z,0) \ar[d]^p \\
      \{0\} \ar@{^{(}->}[r] & 
      (Y,0)
    }
  \]
  is an ICIS on $(Z,0)$.
\end{proposition}

\begin{proof}
  Counting dimensions yields
  \begin{eqnarray*}
    & & \dim (Z,0) - \dim (X_0,0) \\
    &=&  (m+n)(t-1) - (t-1)^2 + N - (N - (m-t+1)(n-t+1)) \\
    &=& m\cdot n \\
    &=& \dim (Y,0).
  \end{eqnarray*}
  The space $(Z,0) \cong (M_{m,n}^t,0) \times (\CC^N,0)$ is known to be Cohen-Macaulay 
  and hence $p$ must define a complete intersection.

  As for the isolated singularity: Choose a representative 
  $A : U \subset \CC^N \to \Mat(m,n;\CC)$ on some open set $U$ with $A$ transverse 
  to the rank stratification in $\Mat(m,n;\CC)$ on $U \setminus \{0\}$. 
  For $y \in Y$ we set 
  \[
    A_y = A - \mathbf Y(y)\colon U \to Y.
  \]
  Let $x\in X_y = A_y^{-1}(M_{m,n}^t)$ be an arbitrary point, 
  $\Sigma^s = A_y^{-1}(S_{m,n}^s)$ 
  the stratum of $X_y$ containing $x$, and $Q = A_y(x) \in S_{m,n}^s$
  its image.
  Since $S_{m,n}^s$ is smooth of codimension 
  $c(s) = (m-s+1)(n-s+1) \leq N$, we may choose a holomorphic function
  \[ 
    h \colon (\Mat(m,n;\CC),Q) \to (\CC^{c(s)},0)
  \]
  such that $(S_{m,n}^s,Q) = (h^{-1}(\{0\}),Q)$. The map $A_y$ being transverse 
  to $S_{m,n}^s$ at $x$ is equivalent to 
  \begin{equation}
    \rank \D (h\circ A_y)(x) = c(s).
    \label{eqn:TransversalityConditionAsRank}
  \end{equation}
  On the other hand, the equations for the stratum $\tilde S_{m,n}^s$ of $Z$ 
  at the point $z=(x,y) \in U \times Y$ are 
  \[
    (\tilde S_{m,n}^s,z) = \{ h( \mathbf Y(y) - A(x) ) = 0 \}.
  \]
  The tangent space to $\tilde S_{m,n}^s$ at $z$ is the kernel of the jacobian
  \[
    T_z \tilde S_{m,n}^s = \ker \D (h\circ (\mathbf Y - A))(z) 
    = \ker \left( \frac{\partial h}{\partial \mathbf Y} \D \mathbf Y 
    - \frac{\partial h}{\partial \mathbf Y} \frac{\partial A}{\partial x} \D x \right).
  \]
  Since 
  \[
    \D \left( h\circ A_y\right)(x) = 
    \frac{\partial h}{\partial \mathbf Y} \frac{\partial A_y}{\partial x} \D x 
    = \frac{\partial h}{\partial \mathbf Y} \frac{\partial A}{\partial x} \D x,
  \]
  it is easy to see that (\ref{eqn:TransversalityConditionAsRank}) is 
  equivalent to $p : (x,y) \mapsto y$ being a submersion on $\tilde S_{m,n}^s$, 
  and hence on $Z$, at $(x,y)$.
  In particular, this holds for all points on $X_0 \setminus \{0\} \subset Z$ 
  and therefore $p$ defines an ICIS.
\end{proof}

We will fix some notation. Let 
\[
  \mathbf W = \left( \{0\} = W_0 \subsetneq W_1 \subsetneq W_2 \subsetneq \cdots 
  \subsetneq W_{m\cdot n-1} \subsetneq W_{m\cdot n} = \CC^{m\cdot n}\right)
\]
be a maximal ascending flag in $Y= \Mat(m,n;\CC)$ and 
\[
  \mathbf V = \left( \Mat(m,n;\CC)\times \CC^N = V_0 \supsetneq V_1 \supsetneq \cdots 
  \supsetneq V_{m\cdot n -1} \supsetneq V_{m\cdot n} \right)
\]
a descending flag in $\Mat(m,n;\CC) \times \CC^N$ with $\dim V_i/V_{i+1} = 1$ 
for each $i$.

For each $k>0$ we set 
\begin{equation}
  Z_k := Z \cap p^{-1}(W_k) \cap V_{k-1}
  \label{eqn:DefinitionZk}
\end{equation}
The two projections from $Z$ induce natural maps 
\begin{equation}
  \xymatrix{ 
    & Z_k \ar[dl]_{f_k} \ar[dr]^{h_k} & \\
    W_k / W_{k-1} & & V_{k}/V_{k+1}
  }
  \label{eqn:DefinitionFkAndHk}
\end{equation}

\begin{proposition}
  \label{prp:ChoiceOfFlags}
  If the flags $\mathbf W$ and $\mathbf V$ are in general position, 
  then the following holds. 
  \begin{enumerate}
    \item Each of the spaces $Z_k$ inherits the canonical Whitney stratification 
      from $(Z,0)$ outside the origin. In particular, $(Z_k,0)$ is a space with isolated 
      singularity if $N < (m-t+1)(n-t+1)$.
    \item Each $f_k$ defines an isolated hypersurface singularity on 
      $(Z_k,0)$ relative to the given stratification. 
    \item The function $h_k$ is a linear equation on $(Z_k,0)$, which 
      can be used to define the complex link and the Carrousel.
  \end{enumerate}
\end{proposition}

\begin{proof}
  We do induction on $k$. Let $k=1$. By definition 
  $V_{k-1} = V_0 = \Mat(m,n;\CC) \times \CC^N$. Consider the map 
  \[
    \PP p : Z\setminus X_0 \to \PP^{m\cdot n-1}, \quad 
    (x,y) \mapsto (p_{1,1}(x,y): \cdots : p_{m,n}(x,y)).
  \]
  We may choose a regular value $[W_1] \in \PP^{m\cdot n-1}$ of this map. 
  Then $Z_1 = p^{-1}(W_1)$ does not have singular points outside 
  $X_0 = \{f_1=0 \} \subset Z_1$. 
  Suppose $(x,0) \in X_0$ was a singular point of $Z_1$ on $X_0$, $x \neq 0$,
  and $S$ the stratum of $Z$ containing it. 
  Write $p = (\tilde p, f_1)$ with 
  \[
    \tilde p \colon z \mapsto z + W_1 \in \Mat(m,n;\CC)/W_1 \cong \CC^{m\cdot n-1}.
  \]
  Being a singular point of $Z_1 = \tilde p^{-1}(\{0\})$ the 
  differential $\D \tilde p|_S(x,0)$ does not have full rank. Then also 
  $\D p|_S(x,0)$ can not have full rank - a contradiction to $X_0$ being 
  an ICIS. We conclude that, being the preimage of a stratified 
  submersion off the origin, $Z_1$ inherits the Whitney stratification from 
  $(Z,0)$ and $f_1 : (Z_1,0) \to \CC$ defines an ICIS on $(Z_1,0)$.

  For a given isolated singularity $f_1 \colon (Z_1,0) \to (\CC,0)$ the 
  condition on a linear equation $h_1$ to be sufficiently general to 
  define the Carrousel is Zariski open (\cite{Tibar95}). 
  We may choose $h_1$ accordingly and 
  set $V_1 = \{ h_1 = 0 \}$.

  For the induction step we start by 
  \[
    \PP \tilde p \colon Z \cap V_k \setminus p^{-1}(W_{k-1}) 
    \to \PP( \Mat(m,n;\CC)/W_{k-1}) , \quad
    (x,y) \mapsto [ p(x,y) + W_{k-1} ].
  \]
  Choose a subspace $W_k \subset \Mat(m,n;\CC)$ with $[W_k/W_{k-1}]$ a 
  regular value of this map. The rest of the 
  induction step is merely a repetition of the above said and left to the reader.
\end{proof}

In what follows, we will from now on assume that the flags $\mathbf V$ and 
$\mathbf W$ have been chosen to fulfill Proposition \ref{prp:ChoiceOfFlags}.
For any $k>0$ let 
\begin{equation}
  F_k = f_k^{-1}(\{\delta\}) \cap Z_k \cap B
  \label{eqn:K-thMilnorFiber}
\end{equation}
be the Milnor fiber of $f_k$ on $Z_k$ for a suitable choice of a Milnor 
ball $B$ and $\delta \in \CC \setminus \{0\}$ small enough. 
We denote the complex link of $Z_k$ by 
\begin{equation}
  L_k = h_k^{-1}(\{\eta\}) \cap Z_k \cap B,
  \label{eqn:K-thComplexLink}
\end{equation}
$\eta \in \CC\setminus\{0\}$ small enough. 
\begin{remark}
  \label{rem:StratificationFk}
  Following the proof of Proposition \ref{prp:ChoiceOfFlags}, it is clear that
  each $F_k$ inherits
  the natural stratification from its embedding in $(Z,0)$. In particular, 
  $F_k$ is smooth if $N < (m-t+2)(n-t+2)$ and otherwise there are 
  strata $\Sigma^s$ for each $s$, 
  \[
    \min \{ r \in \NN : (m-r+1)(n-r+1) \leq N \} \leq s \leq t
  \]
  such that the normal slice to $\Sigma^s$ in $F_k$ is isomorphic 
  to $(M_{m-s+1,n-s+1}^{t-s},0)$. The same holds for $L_k$. Note 
  that unlike the formula (\ref{eqn:RangeOfNonemptyStrata}), we now 
  have ``$\leq$'' rather than ``$<$'', since we also allow zero-dimensional 
  strata on $F$.
\end{remark}
This finishes the exhibition of 
the standard transformation and the associated notation.

\subsection{The induction argument}

We can apply the Handlebody Theorem of Tib{\u a}r at each step $k$ 
in the setup of the previous section to obtain our Main Theorem. 
The key lemma for this induction can already be extracted from 
\cite[Corollary 4.2]{Tibar95}:

\begin{lemma}
  In the final setup of the standard transformation we have 
  for each $0<k<m\cdot n$ a (non-canonical) homeomorphism
  \begin{equation}
    L_k \cong F_{k+1}.
    \label{eqn:IdentificationLinkMilnorFiber}
  \end{equation}
  \label{lem:IdentificationLinkMilnorFiber}
\end{lemma}

\begin{proof}
  One has homeomorphisms 
  \begin{eqnarray*}
    F_{k+1} &=&  Z_{k+1} \cap f_{k+1}^{-1}(\{\delta\}) \cap B \\
    &=& Z \cap V_{k} \cap p^{-1}(W_{k+1}) \cap f_{k+1}^{-1}(\{\delta\}) \cap B \\
    &=& Z \cap V_{k-1} 
      \cap p^{-1}(W_{k+1}) 
      \cap h_k^{-1}(\{0\}) 
      \cap f_{k+1}^{-1}(\{\delta\}) \cap B\\
    &\cong& Z \cap V_{k-1} 
      \cap p^{-1}(W_{k+1}) 
      \cap h_k^{-1}(\{\eta\}) 
      \cap f_{k+1}^{-1}(\{\delta\}) \cap B\\
    &\cong& Z \cap V_{k-1} 
      \cap p^{-1}(W_{k+1}) 
      \cap h_k^{-1}(\{\eta\}) 
      \cap f_{k+1}^{-1}(\{0\}) \cap B\\
    &\cong& Z \cap V_{k-1} \cap p^{-1}(W_k) \cap h_k^{-1}(\{\eta\}) \cap B \\
    &\cong& L_k
  \end{eqnarray*}
  for a Milnor ball $B$ and sufficiently small values for $\delta$ and $\eta$. 
  The homeomorphisms are induced from the parallel transport in the fibration 
  given by 
  \[
    \Phi = (h_k, f_{k+1}) \colon Z \cap V_{k-1} \cap p^{-1}(W_{k+1}) \cap B
    \to \CC \times \CC
  \]
  as in (\ref{eqn:LeFibration})
  over suitably chosen paths connecting $(0,\delta)$, $(\eta,\delta)$, 
  and $(\eta,0)$. 
\end{proof}

\begin{proof}{(of Theorem \ref{thm:MainTheorem})}
  After applying the standard transformation we obtain for $k=1$: 
  \[
    \overline X_u = f_1^{-1}(\{\delta\} \cap Z_1 \cap B = F_1,
  \]
  because $W_1$ was in general position.
  According to the Handlebody Theorem \cite{Tibar95}, this space has 
  a bouquet decomposition
  \[
    F_1 \cong_{ht} L_1 \vee 
    \bigvee_{s_0 \leq s \leq t} \bigvee_{i=1}^{r_1(s)} 
      S^{N-(m-s+1)(n-s+1)+1}( L_{m-s+1,n-s+1}^{t-s,1} )
  \]
  The terms $S^{N-(m-s+1)(n-s+1)+1}( L_{m-s+1,n-s+1}^{t-s,1} )$ stem from 
  the stratified Morse singularities of $h_1$ on $F_1$. These lay in the 
  strata $\Sigma^s$ of different dimension in  $F_1$, 
  cf. Remark \ref{rem:StratificationFk}.
  Since the normal data along 
  each stratum of $F_1$ is determinantal 
  and $h_1$ has only stratified 
  Morse singularities on $F_1$, we may apply Proposition 
  \ref{prp:ThimbleForDeterminantalSingularities}. According to the Handlebody 
  Theorem \ref{thm:HandlebodyTheorem}, the attaching map for each thimble 
  is homotopy equivalent to a one-point-map. Consequently, the attachment of 
  the cone 
  \[
    C\left( S^{N-(m-s+1)(n-s+1)}(L_{m-s+1,n-s+1}^{t-s,1}) \right)
  \]
  along $L_{m-s+1,n-s+1}^{t-s,1}$
  is homotopy equivalent to taking the wedge sum with the suspension 
  \[
    S^{N-(m-s+1)(n-s+1)+1}(L_{m-s+1,n-s+1}^{t-s,1}).
  \]
  We may now proceed inductively and replace $L_k$ by $F_{k+1}$ in this 
  formula according to Lemma \ref{lem:IdentificationLinkMilnorFiber}.
  At each step we attach a certain number $r_k(s)$ of thimbles and we may add 
  them up to $r(s) = \sum_{k=1}^{m\cdot n-1} r_k(s)$.
  This finishes the proof.
\end{proof}

\begin{corollary}
  \label{cor:ToplogyOfSmoothableEIDS}
  If the singularity $(X_0,0)$ in the setting of Theorem \ref{thm:MainTheorem} 
  is smoothable (i.e. if $N < (m-t+2)(n-t+2)$), then 
  \begin{equation}
    \overline X_u \cong_{ht} L_{m,n}^{t,m\cdot n - N} \vee \bigvee_{i=1}^r S^d,
    \label{eqn:BouquetDecompositionSmoothable}
  \end{equation}
  with $d = N - (m -t +1)(n-t+1) = \dim(X_0,0)$.
\end{corollary}

\begin{example}
  Consider determinantal singularities $(X_0,0) \subset (\CC^N,0)$ 
  of type $(2,n,2)$. Let 
  \[
    d = N - (2-2+1)(n-2+1) = N - n +1
  \]
  be the dimension $\dim(X_0,0)$. 
  The inequality characterising smoothability
  becomes 
  \[
    N < (2-2+2)(n - 2 + 2) = 2n.
  \]
  The complex links of the associated generic determinantal 
  variety can easily be computed via the so-called ``Tjurina transform in family'', 
  see \cite{FKZ15}, \cite{Zach16}, or \cite{Zach17}:
  \[
    L_{2,n}^{2,k} \cong_{ht}
    \begin{cases}
      S^2 & \textnormal{ if } 0< k < n \\
      \bigvee_{i=1}^{e-1} S^1 & \textnormal{ if } k=n \\
      e \textnormal{ points} & \textnormal{ if } k = n+1 \\
      \emptyset & \textnormal{ otherwise, }
    \end{cases}
  \]
  where $e$ is the multiplicity of the generic determinantal variety 
  $M_{2,n}^2$. 

  Which one of the complex links is needed, i.e. which $k$ occurs in the 
  formula (\ref{eqn:BouquetDecompositionSmoothable}), depends on $N$: 
  \[
    k = 2n - N.
  \]
  Therefore one has 
  \[
    \overline X_u \cong_{ht} 
    \begin{cases}
      S^2 \vee \bigvee_{i\in I} S^{d} & \textnormal{ if } n < N<2n \\
      \left(\bigvee_{i=1}^{e-1} S^1\right)\vee\left( \bigvee_{i=1}^r S^{1} \right) & 
          \textnormal{ if } N = n \\
	  e' \textnormal{ points} & \textnormal{ if } N = n-1,
    \end{cases}
  \]
  where $e'$ is the multiplicity of $(X_0,0)$.
\end{example}

\begin{remark}
  The decomposition (\ref{eqn:BouquetDecomposition}) in Theorem 
  \ref{thm:MainTheorem} 
  reduces the question about the homotopy type of a determinantal 
  Milnor fiber to the question about the topology of the spaces 
  $L_{m,n}^{t,k}$ apparing in the formula. 
  In those cases, where all these $L_{m,n}^{t,k}$ themselves
  are homotopy equivalent to a bouquet of spheres, the same 
  holds for the determinantal Milnor fiber.

  Moreover, the numbers $r(s)$ measuring the contributions from 
  critical points on the different strata are invariants of 
  the singularity. Theoretically, it is possible to compute these 
  numbers from the Puiseux expansion of the Cerf-diagram 
  $\Delta$ in the Carrousel \cite[section 1.4]{Tibar95} at each induction step in the 
  proof of Theorem \ref{thm:MainTheorem}. However, it would be 
  appealing to have a concise formula relating these numbers 
  to analytic invariants of the singularity. 

\end{remark}